\documentclass[11pt,a4paper]{article}

\usepackage[centertags]{amsmath}
\usepackage{amsfonts}
\usepackage{amssymb}
\usepackage{amsthm}
\usepackage{epsfig}
\usepackage[utf8]{inputenc} 
\usepackage{setspace}
\usepackage{ae} 
\usepackage{eucal} 
\usepackage[numbers, square]{natbib} 
\usepackage{paralist}
\usepackage[usenames]{color}
\usepackage{setspace}
\usepackage{multicol}

\usepackage[active]{srcltx}
\usepackage{hyperref}

\theoremstyle{plain}
\newtheorem{thm}{Theorem}[section]

\newtheorem{prop}[thm]{Proposition}
\theoremstyle{definition}
\newtheorem{defi}[thm]{Definition}

\theoremstyle{remark}
\newtheorem{exmp}[thm]{Example}
\newtheorem{rem}[thm]{Remark}
\newcommand{\DR}{\mathbb{R}}

 \DeclareMathOperator*{\esssup}{esssup}
\DeclareMathOperator*{\essinf}{essinf}
\DeclareMathOperator*{\argmin}{argmin}

\newcommand{\F}{\mathcal{F}}
\newcommand{\lp}{L^{\infty-}}

\allowdisplaybreaks

\begin{document}

\title{Pathwise Iteration for Backward SDEs}

\author{Christian Bender$^{1}$, Christian G\"artner$^{1}$, and Nikolaus Schweizer$^2$}

\maketitle

\footnotetext[1]{Saarland University, Department of Mathematics,
Postfach 151150, D-66041 Saarbr\"ucken, Germany, {\tt
bender@math.uni-sb.de}; {\tt gaertner@math.uni-sb.de}. \\Financial support by the Deutsche Forschungsgemeinschaft under grant BE3933/5-1 is gratefully acknowledged. } 
\footnotetext[2]{University of Duisburg-Essen, Mercator School of Management, Lotharstr. 65, D-46057 Duisburg, {\tt nikolaus.schweizer@uni-due.de}.}
 
\begin{abstract}
 We introduce a novel numerical approach for a class of stochastic dynamic programs which arise as discretizations of backward stochastic differential equations or semi-linear partial differential 
 equations. Solving such dynamic programs numerically requires the approximation of nested conditional expectations, i.e., iterated integrals of previous approximations.  
 Our approach allows us to compute and 
 iteratively improve upper and lower bounds on the true solution starting from an arbitrary and possibly crude input approximation. 
 We demonstrate the benefits of our approach in a high dimensional financial application.
 \\[0.2cm] \emph{Keywords:} Backward stochastic differential equations, dynamic programming, iterated improvement, Monte Carlo
 \\[0.2cm] \emph{AMS subject classifications:} 65C5, 65C30, 49L20, 93E20, 93E24
\end{abstract}

\section{Introduction}

Developing numerical methods for American option pricing, i.e. the  optimal stopping problem, is one of the most specialized and  developed fields in computational finance. We generalize 
several established numerical tools from optimal stopping to a class of convex stochastic dynamic programming equations. Applications  include 
time discretization schemes for backward stochastic differential equations (BSDEs) or, equivalently \cite{pardoux1992backward}, discretization schemes for semi-linear partial differential equations 
(PDEs) where the nonlinearity is convex (or concave). 
In these problems, the numerical challenge has its origin 
in a high order nesting of conditional expectation operators: The approximation at a given time step depends on iterated integrals over the approximations at all future time steps. The curse of 
dimensionality renders many numerical approaches infeasible in such a setting. This includes naive implementations of (nested) Monte Carlo. When moving from optimal stopping to BSDEs or semi-linear 
PDEs, a further numerical challenge arises as approximating derivatives becomes a necessity.

Our main contribution is a pathwise iteration approach which takes an approximate solution of the dynamic programming equation as an input and then constructs upper and lower confidence bounds on 
the true solution. Iteratively taking the super- and subsolutions  corresponding to these bounds as inputs allows to refine the initial bounds. Thus, even a crude input approximation -- such as a 
constant function -- may suffice to provably pin down the solution of a challenging high-dimensional problem up to a tight confidence interval.

For optimal stopping, such ``primal-dual'' approaches for the construction of upper and lower bounds go back to \cite{longstaff2001valuing,tsitsiklis2001regression,R02,HK04,AB04}. 
The approach was extended to our setting of convex dynamic 
programming equations in \cite{BSZ, BGS}, complementing the information relaxation approach of \cite{BSS10} which provides a generalization from optimal stopping to more general optimization problems. 
Iterative improvement methods of the primal-dual approach have been developed in
\cite{CG} for the upper bounds and \cite{KS} for the lower bounds, building on earlier policy iteration techniques due to \cite{H60,P94}. We simplify and unify their arguments in terms of super- and 
subsolutions and generalize them beyond optimal stopping. 

As a second contribution, we introduce two new methods, a minimization and a modified least-squares Monte Carlo (LSMC) method, for computing approximate solutions to stochastic dynamic programming 
equations. While we primarily use these methods as inputs for our improvement approach, both are of independent interest. In the minimization method, we use the pathwise recursion for the 
construction of upper bounds to minimize over a given family of generic input upper bounds. Methods of this type have been proposed in the stopping literature by \cite{Belo13,DFM}. 
Our modified least-squares Monte Carlo algorithm builds on the regress-later method for optimal stopping \cite{GlYu04} and its generalization to BSDEs, the martingale basis algorithm of 
\cite{BS12}. Both 
approaches replace the true solution by an approximation as linear combination of basis functions, for which some computations can be performed in closed form. Our variant of the method has more modest requirements on what can be calculated 
explicitly (i.e. with negligible error), thus increasing its applicability and flexibility. Unlike in optimal stopping and as observed in \cite{BS12}, a considerable benefit of closed-form 
calculations is 
that they may allow to approximate derivatives without further error when derivatives of the basis functions are available. The goal of our modified LSMC algorithm is to gain flexibility by retaining 
only the availability of closed-form derivatives from these previous methods. 

Both the minimization approach and the iterative improvement operate pathwise, i.e., trajectory by trajectory. Compared to classical LSMC methods \cite{longstaff2001valuing,LGW}, they thus have a 
better scope for massively parallel implementations under memory constraints. See \cite{gobet2015stratified} for a recent contribution which highlights these issues and presents a variation of LSMC 
which is more amenable to parallelization.
We confirm the practical applicability of our methods in a classical reference problem, pricing under funding risk in a financial market model driven by a five-dimensional Markov process. 
Depending on the time discretization, this corresponds to integrating out between 100 and 200 variables with a complex dependence structure in our Monte Carlo approach. 

The paper is organized as follows: Section \ref{sec setup} introduces the setting. Section \ref{sec improvement subsolutions} develops the theory behind our iteration approach for subsolutions, i.e., 
lower bounds, while Section \ref{sec improvement supersolutions} provides the 
analogous results for upper bounds. Section \ref{sec Implementation} provides an overview of our numerical approach, including our new approximation methods. Numerical results in 
the context of funding risk are presented in Section \ref{sec funding}.

\section{Setup}
\label{sec setup}

 Throughout the paper, we study the following type of convex dynamic programming equation on a complete filtered probability space $(\Omega, \F, (\F_j)_{j=0,\ldots J}, P)$ in discrete time:
 \begin{align}
    Y_j &= F_j (E_j [\beta_{j+1} Y_{j+1}]), \quad j=J-1,\ldots, 0, \quad Y_J= \xi, \label{dynprog} 
 \end{align}
 with given data $\xi$, $F$ and $\beta$ (to be specified below), where $E_j [\cdot]$ denotes the conditional expectation with respect to $\F_j$. This type of recursive equation encompasses the dynamic programming equation for optimal stopping 
of an adapted discrete time process $S$ (or, in financial terms, the Bermudan option pricing problem),
\begin{align}\label{eq:stopping}
 Y_j=\max\{S_j,E_j [ Y_{j+1}]\},\quad Y_J=S_J,
\end{align}
 see e.g. \cite{KS}, and discretization schemes for backward stochastic differential equations of the form
\begin{align}\label{eq:BSDE}
 Y_j=E_j [ Y_{j+1}]+ (t_{j+1}-t_j)\, G\left(t_j,E_j [ Y_{j+1}], E_j \left[\frac{W_{t_{j+1}}- W_{t_{j}}}{t_{j+1}-t_j} Y_{j+1}\right] \right),\quad Y_J=\xi,
\end{align}
for given data $\xi$ and $G$. Here, $(t_0,\ldots,t_J)$ denotes a partition of a time interval $[0,T]$, $W$ is a multidimensional Brownian motion (whose increments may be truncated for practical purposes),
and $\mathcal{F}_j$ is the information generated by $W$ up to time $t_j$, see \cite{FTW} for this specific scheme in the more general context of second order BSDEs and \cite{BS12} for a literature overview.
 
 We assume that $F_j : \Omega \times \DR^D \rightarrow \DR$
 is measurable for every $j=0,...,J-1$ and that the process $(j,\omega)\mapsto F_j(\omega,z)$ is adapted for every $z \in \DR^D$.  
 Moreover, for every  
 $j=0,\ldots, J-1$ and $\omega\in \Omega$, the map $z\mapsto F_j(\omega,z)$ is convex in $z$. Additionally, $F_j$ satisfies a (stochastic) polynomial growth condition for every $j=0,...,J-1$, i.e. 
 there exist a constant $q \geq 0$ and an adapted, nonnegative processes $\alpha$, which is in $L^p (\Omega,P)$ for every $p \geq 1$, such that
 \[
  \left| F_j(z) \right| \leq \alpha_j (1+ | z |^q )
 \]
 holds $P$-a.s. for every $z \in \DR^D$. The $\DR^D$-valued process $\beta$ is adapted and in $L^p (\Omega,P)$ for every $p \geq 1$.  
 The terminal condition $\xi$ is an $\F_J$-measurable, $\DR$-valued random variable with $E[|\xi|^p] < \infty$ for all $p \geq 1$. 

 Further, we introduce the following  notation: For $m \in \mathbb{N}$, we denote by $\lp(\mathbb{R}^m)$ the set of $\mathbb{R}^m$-valued random variables that are in $L^p(\Omega,P)$
 for all $p \geq 1$. 
 The set of $\F_j$-measurable random variables that are in $\lp(\DR^m)$ is denoted by $\lp_j(\DR^m)$. In addition, $\lp_{ad} (\DR^m)$ denotes the set of adapted processes $Z$ 
 such that $Z_j \in \lp_j(\DR^m)$ for every $j=0,...,J$.
 From the integrability properties of the terminal condition $\xi$ and the weight process $\beta$ as well as the polynomial growth condition
 on $F$, we deduce by backward induction that the ($P$-a.s. unique) solution $Y$ to \eqref{dynprog} is in $\lp_{ad} (\DR)$.
 
 Super- and subsolutions which are central later on  are defined as follows:
\begin{defi}
  A process $Y^{up} \ (\text{respectively } Y^{low}) \in \lp_{ad}(\DR)$ is called \emph{supersolution} (respectively \emph{subsolution}) to the dynamic program \eqref{dynprog} if 
  $Y_J^{up} \geq Y_J$ (respectively $Y_J^{low} \leq Y_J$) and for every $j=0,\ldots,J-1$ it holds that
  \[
   Y_j^{up} \geq F_j \left( E_j \left[ \beta_{j+1} Y_{j+1}^{up} \right] \right) ,\quad P\textnormal{-a.s.}
  \]
  (and with ' $\geq$' replaced by ' $\leq$' for a subsolution).
\end{defi}

\noindent
In general, we cannot expect that super- and subsolutions $Y^{up}$ and $Y^{low}$ to \eqref{dynprog} are bounds on the true solution $Y$, i.e. we need not have that $Y_j^{up} \geq Y_j \geq Y_j^{low}$ 
holds $P$-a.s. for every $j=0,...,J$.
To ensure this, we impose the following \emph{monotonicity} assumption throughout this paper: For $Y^{(1)}, \ Y^{(2)} \in \lp(\DR)$ with $Y^{(1)} \geq Y^{(2)}$ $P$-a.s. and every $j=0,...,J-1$,
it holds that
\begin{align}\label{monotonicity}
 F_j ( \beta_{j+1} Y^{(1)}) \geq F_j (\beta_{j+1} Y^{(2)}),\quad P\textnormal{-a.s.}
 \end{align}
 Applying Theorem 4.3 in \cite{BGS} twice (first with  the filtration $(\mathcal{G}_j)_{j=0,\ldots,J}=(\mathcal{F})_{j=0,\ldots,J}$, and then with the given filtration $(\mathcal{F}_j)_{j=0,\ldots,J}$), we observe that this monotonicity assumption implies the following \emph{comparison principle}:
 
\begin{prop}
Let $Y^{up}$ and $Y^{low}$ be super- and subsolutions to \eqref{dynprog}. Then, under the given assumptions, it holds that,  for every $j=0,...,J$,
 \[
  Y_j^{up} \geq Y_j^{low} \quad P\text{-a.s.}
 \]
 \label{comparison}
\end{prop}

The improvement algorithms presented in the following sections are based on the primal-dual methodology introduced by \cite{AB04,HK04,R02} in the context of Bermudan option pricing
and further developed in \cite{BSZ} and \cite{BGS} for dynamic programming equations of the form \eqref{dynprog}. This approach relies on 
the choice of suitable martingales and controls which are derived from an approximate solution to \eqref{dynprog} and are used as an input for constructing 
super- and subsolutions. We, therefore, denote by $\mathcal{M}_D$ 
the set of $\DR^D$-valued martingales which are elements of $\lp_{ad} (\DR^D)$.
For a process $Z \in \lp_{ad} (\DR^D)$, we refer to the martingale part of the Doob decomposition of $Z$, which is given by 
\[
 \sum_{i=0}^{j-1} Z_{i+1} - E_i [Z_{i+1}], \quad j=0,...,J,
\]
as \emph{Doob martingale of $Z$}. In particular, we get that the Doob martingale of the process $\beta \bar{Z}$ is in $\mathcal{M}_D$ for any $\bar{Z} \in \lp_{ad}(\DR)$. 
While suitable martingales are the main ingredient of the upper bounds, we derive lower bounds  by rewriting \eqref{dynprog} as a stochastic control problem using 
convex duality techniques. To this end, recall that the convex conjugate of $F_j$ is, for every $\omega \in \Omega$, given by
 \[
  F_j^{\#} \left(\omega,u \right) := \sup_{z \in \DR^D} (u^{\top} z - F_j (\omega,z)),
 \]
 with effective domain 
 \[
  D_{F^\#}^{(j,\omega)} = \left\{ u \in \DR^D \ \left| \ F_j^{\#} (\omega,u) < \infty \right. \right\}. 
 \]
As we will see below, the sets of admissible controls in our problem are given by
\begin{align}
  \mathcal{A}_j = \left\{ \left. \left(r_i \right)_{i=j,\ldots,J-1} \right| r_i \in \lp_i (\DR^D), F_i^{\#} (r_i) \in \lp(\DR) \textnormal{ for }i=j,\ldots,J-1\right\}, 
\end{align}
where $j=0,...,J-1$.

 \section{Improvement of subsolutions}
 \label{sec improvement subsolutions}
 
 In this section we propose an iterative algorithm to improve a given subsolution to \eqref{dynprog}. This approach generalizes in some sense the idea of \cite{KS}, who presented an iterative
 method to improve a given family of stopping times in the context of Bermudan option pricing. We begin this section by recalling  a construction of subsolutions from \cite{BSZ}
and, then, explain how it can be used to improve arbitrary subsolutions. 
 
 In order to construct a subsolution to \eqref{dynprog}, we linearize this dynamic programming equation in the following way: By convexity and closedness of $F_j$, we have due to Theorem 12.2 in \cite{rockafellar} that $F_j^{\#\#} = F_j$ for every $j=0,...,J-1$ and $\omega \in \Omega$. Hence, for
 every $j=0,...,J-1$, $\omega \in \Omega$ and $z \in \DR^D$, it holds that
 \begin{align}
  F_j(\omega,z) = \sup_{u \in \mathbb{R}^D} u^{\top} z - F_j^{\#}(u). \label{linearization}
 \end{align}
 From Lemma A.1 in \cite{BGS}, we get existence of an adapted process $r^* \in \mathcal{A}_0$ which solves 
 \begin{align}
  \left(r_i^*\right)^{\top} E_i[\beta_{i+1} Y_{i+1}] - F_i^{\#} (r_i^*) &= F_i (E_i [\beta_{i+1} Y_{i+1}]) \quad P\text{-a.s.}, \label{optimal control}
 \end{align}
 for every $i=0,...,J-1$.
 We now define the typically non-adapted process $\theta^{low}$ as in Remark 3.6 (ii) in \cite{BSZ}. To this end, we fix a martingale $M \in \mathcal{M}_D$
 and an admissible control $r \in \mathcal{A}_0$. Then, the pathwise recursion for $\theta^{low} := \theta^{low} (r,M)$ is, in our notation, given as follows:
 \begin{align}
   \theta_j^{low} &= r_j^{\top} \beta_{j+1} \theta_{j+1}^{low} -r_j^{\top} \Delta M_{j+1} - F_j^{\#} (r_j), \quad j=J-1,...,0, \quad  \theta_J^{low} = \xi, \label{recursion low}
 \end{align}
 where $\Delta M_{j+1} := M_{j+1}-M_j$.
 By backward induction, we get that $\theta_j^{low} \in \lp(\DR)$ for every $j=0,...,J$, since $r$ and $\beta$ are in $\lp_{ad}(\DR^D)$ and $M \in \mathcal{M}_D$ by assumption. 
 Hence, we can define a subsolution $Y^{low}$ by $Y_j^{low} := E_j [\theta_j^{low} ]$ for every $j=0,...,J$. Indeed, by the tower property of the conditional 
 expectation and \eqref{linearization}, we observe that
 \begin{align*}
 Y_j^{low} = r_j^{\top} E_j \left[\beta_{j+1} \theta_{j+1}^{low} \right] - F_j^{\#} (r_j) \leq F_j \left(E_j \left[\beta_{j+1} Y_{j+1}^{low} \right] \right)
 \end{align*}
 holds and, hence, by Proposition \ref{comparison}, we conclude that
 $Y_j \geq Y_j^{low}$ $P$-a.s., for any $j=0,...,J$. 
 Moreover, \cite{BSZ} prove that the solution $Y$ to \eqref{dynprog} is the value of a primal maximization problem, i.e.
 \[
  Y_j = \esssup_{r \in \mathcal{A}_j} E_j [\theta_j^{low} (r,M)] \quad P \textnormal{-a.s.}, \quad j=0,...,J.
 \]
 Indeed, every control $r^* \in \mathcal{A}_j$ which satisfies   \eqref{optimal control} for $i=j,\ldots, J-1$ achieves the maximum. 
 We emphasize that the expression $E_j [\theta_j^{low}(r,M)]$ is independent of $M$, as the martingale is only a control variate in the recursion for $\theta^{low}$ and thus vanishes by taking conditional expectation.
 However, a straightforward computation shows that the Doob martingale $M^*$ of $\beta Y$ acts as a perfect control variate in the case of optimal controls, i.e. for every $j=0,...,J$ it holds that, 
 \begin{align}\label{eq:optimal_variate}
  \theta_j^{low} (r^*,M^*) = Y_j \quad P \textnormal{-a.s.} 
 \end{align}

 \subsubsection*{Iterative improvement of subsolutions}
 
 Suppose we are given an arbitrary subsolution $\bar{Y}$. We next show that the construction of the process $\theta^{low} (r,M)$ in  \eqref{recursion low}, implies an improvement of the subsolution 
 $\bar{Y}$ in the sense that for suitable choices of $r \in \mathcal{A}_0$ and  $M \in \mathcal{M}_D$ we have
 \[
  Y_j \geq E_j \left[ \theta_j^{low} (r,M) \right] \geq \bar{Y}_j \quad P \textnormal{-a.s.}
 \]
 for every $j=0,...,J$. The subsolution $(E_j [\theta_j^{low} (r,M)])_{j=0,...,J}$ is then called an \emph{improvement} of the subsolution $\bar{Y}$. Theorem \ref{improvement subsolution} below, 
 explains how to construct such an improvement. Further, we show that our construction only gets stuck if the subsolution $\bar{Y}$, which we want to improve, already coincides 
 with the solution $Y$ to \eqref{dynprog}.
 
 \begin{thm}
  Let $j\in \{0,...,J-1\}$, let $\bar{Y}$ be a subsolution to \eqref{dynprog} and denote by $\bar{M} \in \mathcal{M}_D$ the Doob martingale of $\beta \bar{Y}$.
  Further let $\bar{r} \in \mathcal{A}_0$ be an adapted process that solves 
  \begin{align}
   \bar{r}_i^{\top} E_i \left[\beta_{i+1} \bar{Y}_{i+1} \right] - F_i^{\#}(\bar{r}_i) = F_i \left(E_i \left[\beta_{i+1} \bar{Y}_{i+1} \right] \right) \quad P \textnormal{-a.s.} \label{equation policy}
  \end{align}
  for every $i=0,...,J-1$. Then, for any $M \in \mathcal{M}_D$, $\theta^{low} (\bar{r},{M})$ defined by \eqref{recursion low} satisfies
  \begin{align}
   Y_i \geq E_i \left[ \theta_i^{low}(\bar{r},{M}) \right] \geq F_i \left(E_i \left[\beta_{i+1} \bar{Y}_{i+1} \right] \right) \geq \bar{Y}_i \quad P \textnormal{-a.s.}, \label{improvement subsolution inequalities}
  \end{align}
  for all $i=0,...,J-1$.
  Moreover, if $\bar{Y}_i = Y_i$ for all $i=j+1,...,J$, then 
  \[
  E_j \left[ \theta_j^{low}(\bar{r},{M})\right]= \theta_j^{low}(\bar{r},\bar{M}) = Y_j \quad P \textnormal{-a.s.}
  \]
  \label{improvement subsolution}
 \end{thm}
 
 \begin{proof}
  As we have seen above, the process $(E_j [\theta_j^{low} (r,M)])_{j=0,...,J}$ defines a subsolution for any martingale $M \in \mathcal{M}_D$ and $r \in \mathcal{A}_0$, so that the first inequality 
  in \eqref{improvement subsolution inequalities} is already shown. The last inequality in  \eqref{improvement subsolution inequalities} is immediate, as $\bar Y$ is assumed to be a subsolution.
To prove the remaining inequality in  \eqref{improvement subsolution inequalities}, we denote $\theta^{low}=\theta^{low}(\bar r, \bar M)$.
Recalling that $E_j [\theta_j^{low} (r,M)]$ does not depend on the choice of $M \in \mathcal{M}_D$, it suffices to show that
\[
   \theta_i^{low} \geq F_i \left(E_i \left[\beta_{i+1} \bar{Y}_{i+1} \right] \right)  \quad P \textnormal{-a.s.}
  \]
  by backward induction on $i$. The assertion then follows by the monotonicity of the conditional expectation. The case $i=J$ is trivial, since we have
  $
   \theta_J^{low} = \xi \geq \bar{Y}_J,
  $
  by definition.
  Now suppose that the assertion is true for $i+1 \in \{1,...,J\}$, and, thus, we have $\theta_{i+1}^{low} \geq \bar{Y}_{i+1}$ $P$-a.s.  
  Then, it follows from the definition of $\bar{M}$, \eqref{equation policy}, and the induction hypothesis that 
  \begin{eqnarray*}
   \theta_i^{low}  &=& \bar{r}_i^{\top} (\beta_{i+1} \theta_{i+1}^{low} - (\beta_{i+1} \bar{Y}_{i+1} -E_i [\beta_{i+1} \bar{Y}_{i+1}])) - F_i^{\#}(\bar{r}_i) \\
		   &=& \bar{r}_i^{\top} \beta_{i+1} (\theta_{i+1}^{low} - \bar{Y}_{i+1}) + F_i(E_i [\beta_{i+1} \bar{Y}_{i+1}]) \geq F_i (E_i [\beta_{i+1} \bar{Y}_{i+1}]).
  \end{eqnarray*}
  Here, the inequality is a consequence of the induction hypothesis and the positivity of $\bar{r}_i^{\top} \beta_{i+1}$ which is due to the monotonicity assumption \eqref{monotonicity}
  on the function $F_i$, see Theorem 4.3 in \cite{BGS} (with the constant full information filtration $(\mathcal{G}_j)_{j=0,\ldots,J}=(\mathcal{F})_{j=0,\ldots,J}$).
  
  To complete the proof, we assume that $\bar{Y}_i = Y_i$ for all $i=j+1,...,J$. 
  Hence, we observe by \eqref{equation policy}, that $\bar{r}_i$ satisfies the optimality condition \eqref{optimal control} 
  $P$-a.s. for every $i=j,...,J$.  
  Further, we conclude by the definition of $\bar{M}$ that the increments $\bar{M}_{i+1}- \bar{M}_i$ coincide with the ones of the Doob martingale of 
$\beta Y$ for $i=j,...,J-1$.
  By \eqref{eq:optimal_variate}, we thus obtain that 
  $
   \theta_j^{low} = Y_j 
  $
  $P$-a.s., which completes the proof.
 \end{proof}
 
 When starting with an arbitrary subsolution, we typically do not obtain the solution $Y$ by applying
 the approach described in Theorem \ref{improvement subsolution} once. However, this construction can be iterated in a straightforward way:
 Let ${Y}^{(low,0)}$ be a subsolution and define $\theta^{(low,0)} := {Y}^{(low,0)}$. We define the $k$-th iteration according to \eqref{recursion low} by 
 \begin{align}
   \theta^{(low,k)} := \theta^{low} (r^{(k)},M^{(k)}),\quad k \geq 1,  	\label{iteration subsolution}
 \end{align}
 where 
the process $r^{(k)} \in \mathcal{A}_0$ is for every $j=0,...,J$ given by
 \begin{align}
  \left( r_j^{(k)} \right)^{\top} E_j \left[ \beta_{j+1} \theta_{j+1}^{(low,k-1)} \right] - F_j^{\#} \left(r_j^{(k)} \right) = F_j \left( E_j \left[\beta_{j+1} \theta_{j+1}^{(low,k-1)} \right] \right), 
  \label{iteration subsolution control}
 \end{align}
and $M^{(k)} \in \mathcal{M}_D$ is arbitrary. Applying Theorem \ref{improvement subsolution} iteratively, we observe that $ E_j[\theta_j^{(low,k)}]\geq 
 E_j[\theta_j^{(low,k-1)}]$, $P$-a.s, for every $k\geq 1$ and $j=0,\ldots, J$. Moreover,
\[
   E_i\left[\theta_i^{(low,J-j)}\right] = Y_i \quad P \textnormal{-a.s.},
  \]
whenever $i\geq j$. In the last equation, the conditional expectation on the left-hand side can be removed, when each $M^{(k)}$ is taken as the Doob martingale of 
$\beta_j E_j[\theta^{(low,k-1)}_j]$. We, thus, observe that $Y$ is the $P$-a.s. unique fixed point of this 
 iteration, which actually terminates after at most $J$ iteration steps.

 \subsubsection*{Improvement of a family of subsolutions}
 
 In Section \ref{sec Implementation} below, we explain that the numerical costs of algorithms based on \eqref{iteration subsolution} tend to grow exponentially in the number of 
 iterations $k$. For this reason, a moderate number of iterations must suffice in practical implementations. One way to address this issue is to improve a whole family of subsolutions simultaneously 
 instead of just one subsolution. 
 To this end, let $(\bar{Y}^{\{l\}})_{l \in I}$ be a family of subsolutions, where $I$ is a finite index set.  
 Further, we denote by $K(j)$, $j=1,...,J$, a nondecreasing sequence of subsets of $I$, i.e. it holds that $K(j) \subset K(j+1)$. 
 Then, we consider the predictable, $I$-valued process 
  \[
   l^*(j) =  \inf \left\{ l \in K(j)  \left| \ \forall \iota \in K(j) \ F_{j-1} \left(E_{j-1} \left[\beta_j \bar{Y}_j^{\{l\}}\right] \right. \right) 
	    \geq F_{j-1} \left(E_{j-1} \left[\beta_j \bar{Y}_j^{\{\iota\}}\right] \right) \right\}.	\label{def lstern}
   \]
 This means that, at every time point $j=1,...,J$, we only consider those subsolutions which are represented in the subset $K(j)$ and the random variable $l^*(j)$ returns an index $l\in K(j)$
at which the evaluation of $F_{j-1}$ is maximized. In the simplest case $K(j) = I$ for all $j=1,...,J$. More sophisticated choices of $K(j)$ allow to reduce the computational costs to determine 
$l^*$.
 We claim that the process $\bar{Y}$ which is given by 
  \begin{align}\label{eq:familyimprovement}
   \bar{Y}_j= \bar{Y}_j^{\{l^*(j)\}} {\bf 1}_{\{ j>0\}}+  F_0 \left(E_0 \left[\beta_1 \bar{Y}_1 \right] \right){\bf 1}_{\{ j=0\}}
  \end{align}
 is a subsolution to \eqref{dynprog}, which allows us to improve the subsolutions $(\bar{Y}^{\{l\}})_{l \in I}$ simultaneously. To examine the subsolution property of $\bar{Y}$, we first observe that
 the case $j=0$ is trivial, since we have $\bar{Y}_0 = F_0 (E_0 [\beta_1 \bar{Y}_1])$ by definition. For the case $j>0$, we get by the 
 subsolution property of $\bar{Y}^{\{l\}}$ for every $l\in I$, and as $K(j) \subset K(j+1)$, that
  \begin{align*}
   \bar{Y}_j   &= \sum_{l \in K(j)} \bar{Y}_j^{\{l\}} 1_{\{l^*(j) = l \}} \leq \sum_{l \in K(j)} F_j \left(E_j \left[\beta_{j+1} \bar{Y}_{j+1}^{\{l\}} \right] \right) 1_{\{l^*(j) = l \}} \\
	       &\leq \sum_{l \in K(j)} F_j \left(E_j \left[\beta_{j+1} \bar{Y}_{j+1}^{\{l^*(j+1)\}} \right] \right) 1_{\{l^*(j) = l \}} = F_j \left(E_j \left[\beta_{j+1} \bar{Y}_{j+1} \right] \right) \quad P \textnormal{-a.s.}
  \end{align*}
  Hence, Theorem \ref{improvement subsolution} can be applied to the process $\bar{Y}$ and implies, for $\theta^{low}=\theta^{low} (\bar{r},M)$,
  \begin{align}\label{eq:familyimprovement2}
   E_j \left[\theta_j^{low} \right] \geq F_j \left(E_j \left[\beta_{j+1} \bar{Y}_{j+1} \right] \right) 
       = \max_{l \in K(j+1)} F_j \left(E_j \left[\beta_{j+1} \bar{Y}_{j+1}^{\{l\}} \right] \right) \geq \max_{l \in K(j+1)}\bar{Y}_j^{\{l\}}
  \end{align}
  $P$-a.s. for all $j=0,...,J-1$, where $\bar{r}$ is for every $j=0,...,J-1$ given by \eqref{equation policy} and where $M \in \mathcal{M}_D$. 
  Thus, if $K(j)=I$ for all $j=1,...,J$, we achieve a simultaneous improvement of all subsolutions $(\bar{Y}^{\{l\}})_{l \in I}$ by improving $\bar{Y}$.

\begin{exmp}
 We consider the optimal stopping problem of an adapted process $S$, whose value process $Y$ (the so-called Snell envelope) is governed by \eqref{eq:stopping}. Suppose
$(\tau_l)_{l\in I}$, where $I=\{0,\ldots, J\}$, is a family of 
$I$-valued stopping times which is consistent in the sense of \cite{KS}: For every $l\in I$
$$
\tau_l \geq l\quad \textnormal{ and } \quad \left(\tau_l >l \; \Rightarrow \; \tau_l=\tau_{l+1} \right).
$$
This consistency condition implies  that each of the processes $\bar{Y}^{\{l\}}_j:=E_j[S_{\max\{\tau_l,\tau_j\}}]$, $l\in I$, defines a subsolution to  \eqref{eq:stopping}.
Define  $K(j)=\{0,\ldots, \min\{j+\kappa-1,J\}\}$ for some window parameter $\kappa \in \mathbb{N}$.  Specializing
the improvement condition \eqref{equation policy} based on the subsolution constructed in \eqref{eq:familyimprovement} to the optimal stopping problem, one can verify that the improved subsolution 
satisfies, for every control variate $M\in \mathcal{M}_D$, $E_j[\theta^{low} (\bar{r},M)]=E_j[S_{\bar{\tau}_j}]$ where the stopping times $\bar{\tau}_j$ are given by
$$
\bar{\tau}_j=\inf\{i\geq j;\; S_i\geq \max_{l=i+1,\ldots, \min\{i+\kappa, J\}} E_i[S_{\tau_l} ]\},\quad j=0,\ldots,J.
$$ 
Hence, for every $j=0,\ldots,J$, by \eqref{eq:familyimprovement2} and the consistency condition,
$$
E_j[S_{\bar{\tau}_j}]\geq \max_{l=j+1,\ldots, \min\{j+\kappa, J\}}  \max\{E_j[S_{\tau_l} ], S_j\} \geq \max_{l=j,\ldots, \min\{j+\kappa, J\}} E_j[S_{\tau_l} ]  , \quad P\textnormal{-a.s.} 
$$
Thus, we recover the policy improvement result in Theorem 3.1 of \cite{KS} as a special case of our approach.
\end{exmp}

\section{Improvement of supersolutions}
\label{sec improvement supersolutions}

In this section, we propose an iterative way for improving supersolutions to convex dynamic programs like \eqref{dynprog}. We generalize the construction of \cite{CG}, who presented an 
improvement approach for supersolutions in the context of optimal stopping. Similar to Section \ref{sec improvement subsolutions}, we build our
approach on the pathwise recursion for upper bounds presented in \cite{BSZ}. Therefore, we begin this section with a brief overview of their construction and explain how it can be applied for improving
arbitrary supersolutions. The remainder of this section is dedicated to transferring the results obtained in Section \ref{sec improvement subsolutions} to supersolutions.

The main idea of the pathwise approach presented in \cite{BSZ} is to remove the appearing conditional expectations in \eqref{dynprog} and, instead, subtract a martingale 
increment. More precisely, let $M \in \mathcal{M}_D$ be a martingale. Then, we define the typically non-adapted process $\theta^{up}:=\theta^{up}(M)$ by
\begin{align}
  \theta_j^{up} &= F_j (\beta_{j+1} \theta_{j+1}^{up} - \Delta M_{j+1}),\quad j=0,...,J-1 ,\quad \theta_J^{up} = \xi.  \label{recursion up}
\end{align}
 Due to the polynomial growth condition on $F_j$ and the integrability properties of $\beta$ and $M$, we get by backward induction that $\theta_j^{up} \in \lp(\DR)$ for every 
$j=0,...,J$.
Setting $(Y_j^{up})_{j=0,...,J} = (E_j [\theta_j^{up}])_{j=0,...,J}$, we observe immediately by Jensen's inequality, the martingale property of $M$, and the 
tower property of the conditional expectation that $Y^{up}$ is a supersolution to \eqref{dynprog}:
\begin{align*}
  Y_j^{up} &= E_j [F_j (\beta_{j+1} \theta_{j+1}^{up} - \Delta M_{j+1})] \geq  F_j ( E_j [\beta_{j+1} Y_{j+1}^{up}]).
\end{align*}
Similarly to the results in Section \ref{sec improvement subsolutions}, \cite{BSZ} show that the solution $Y$ to \eqref{dynprog} can be represented as a dual minimization problem, i.e.
\begin{align}\label{mini_prob}
 Y_j = \essinf_{M \in \mathcal{M}_D} E_j [\theta_j^{up}(M)] \quad P \textnormal{-a.s.}
\end{align}
for every $j=0,...,J$. The Doob martingale $M^*$ of $\beta Y$ achieves the minimum and additionally is even pathwise optimal, i.e.  for every $j=0,...,J$,
\begin{align}\label{eq:optimal_martingale}
 Y_j = \theta_j^{up} (M^*) \quad P \textnormal{-a.s.}
\end{align}
 
\subsubsection*{Iterative improvement of supersolutions}

As in Section \ref{sec improvement subsolutions}, we show that an arbitrary supersolution $\bar{Y}$ can be improved in the sense that for a suitable martingale 
$M \in \mathcal{M}_D$ and the process $\theta^{up}(M)$ from \eqref{recursion up} the inequality 
\begin{equation*}
    Y_j \leq E_j [\theta_j^{up}(M)] \leq \bar{Y}_j
\end{equation*}
holds $P$-a.s. for every $j=0,...,J$. In the context of optimal stopping, \cite{CG} show that taking the Doob martingale of a given supersolution, leads to an improvement.
Theorem \ref{improvement supersolution} generalizes this idea to convex dynamic programs of the form \eqref{dynprog}:

\begin{thm}
  Let $j\in \{0,...,J-1\}$ and let $\bar{Y}$ be a supersolution to \eqref{dynprog}. Further, let $\bar{M}$ be the Doob martingale of the process $\beta \bar{Y}$.
  Then, the process $\theta^{up}(\bar{M})$  satisfies
  \begin{align}
    Y_i \leq E_i [\theta_i^{up}(\bar M)] \leq F_i (E_i [\beta_{i+1} \bar{Y}_{i+1}]) \leq \bar{Y}_i \quad P \textnormal{-a.s.} \label{improvement supersolution inequalities}
  \end{align}
  for all $i=0,...,J$.
  Moreover, if $\bar{Y}_i = Y_i$ for all $i=j+1,...,J$, then 
  \[
   \theta_j^{up}(\bar M) = Y_j \quad P \textnormal{-a.s.}
  \]
  \label{improvement supersolution}
\end{thm}

\begin{proof}
  The overall strategy of proof is similar to the one of Theorem \ref{improvement subsolution}.
  At the beginning of this section, we have already shown that $(E_j [\theta_j^{up}(M)])_{j=0,...,J}$ is a supersolution for any martingale $M \in \mathcal{M}_D$, which yields the first
 inequality
  in \eqref{improvement supersolution inequalities}. The last one is due to the supersolution property of $\bar Y$.
   To show the remaining inequality, we prove again the slightly stronger assertion
  \[
   \theta_i^{up}:=\theta_i^{up}(\bar M) \leq F_i (E_i [\beta_{i+1} \bar{Y}_{i+1}])  \quad P \textnormal{-a.s.}, \ i=0,...,J-1,
  \]
  from which we obtain \eqref{improvement supersolution inequalities} by the monotonicity of the conditional expectation.
  The proof is by backward induction on $i$, with the case $i=J$ being trivial, since, by definition, we have $\theta_J^{up} = \xi \leq \bar{Y}_J$.
  Now suppose that the assertion is true for $i+1 \in \{1,...,J\}$, i.e. $\theta_{i+1}^{up} \leq \bar{Y}_{i+1}$ $P$-a.s. 
  Hence, we conclude by the definition of $\bar{M}$, the monotonicity assumption \eqref{monotonicity}, and the induction hypothesis that 
  \begin{eqnarray*}
   \theta_i^{up}  &=& F_i (\beta_{i+1} \theta_{i+1}^{up} - (\beta_{i+1} \bar{Y}_{i+1} - E_i [\beta_{i+1} \bar{Y}_{i+1}])) \\
		  &\leq& F_i (\beta_{i+1} \bar{Y}_{i+1} - (\beta_{i+1} \bar{Y}_{i+1} - E_i [\beta_{i+1} \bar{Y}_{i+1}])) = F_i (E_i [\beta_{i+1} \bar{Y}_{i+1}]). 
		    \end{eqnarray*}
Here, we  exploit that   $z\mapsto F_i (z - (\beta_{i+1} \bar{Y}_{i+1} - E_i [\beta_{i+1} \bar{Y}_{i+1}]))$ inherits the monotonicity property \eqref{monotonicity} from $F_i$
by the equivalent characterization of the monotonicity property 
via positivity in Theorem 4.3 of \cite{BGS} (with the constant full information filtration),
because its convex conjugate is given by $F_i^\#(u)+ u^\top(\beta_{i+1} \bar{Y}_{i+1} - E_i [\beta_{i+1} \bar{Y}_{i+1}])$.

To complete the proof, we assume that $\bar{Y}_i = Y_i$ for all $i=j+1,...,J$, where $j \in \{0,...,J-1\}$ is fixed. 
Then, again, the increments  $\bar{M}_{i+1}-\bar{M}_i$ coincide with those of the Doob martingale $M^*$ of $\beta Y$ for $i=j,\ldots, J-1$. 
Hence, \eqref{eq:optimal_martingale} concludes.
\end{proof}
 
 As in Section \ref{sec improvement subsolutions}, this improvement can be iterated several times. For a given supersolution $Y^{(up,0)}$ define $\theta^{(up,0)} := Y^{(up,0)}$ and
 define $\theta^{(up,k)}$ according to \eqref{recursion up} by
 \begin{align}
   \theta^{(up,k)} := \theta^{up} (M^{(k)}),\quad k\geq 1,
   \label{iteration supersolution}
 \end{align}
 where each $M^{(k)}$ is given by
 \begin{align}
  M_j^{(k)} = \sum_{i=0}^{j-1} \beta_{i+1} E_{i+1} \left[\theta_{i+1}^{(up,k-1)} \right] - E_i \left[\beta_{i+1} \theta_{i+1}^{(up,k-1)} \right], \quad j=0,...,J.
  \label{iteration supersolution martingale}
 \end{align}
Then, iterative application of Theorem  \ref{improvement supersolution} yields $E_j[\theta_j^{(up,k)}]\leq E_j[\theta_j^{(up,k-1)}]$, $P$-a.s., for every $k\geq 1$ and $j=0,\ldots, J$, and 
\[
     \theta_i^{(up,J-j)} = Y_i, \quad P \textnormal{-a.s.},
    \]
whenever $i\geq j$.  So the upper bound iteration also terminates after at most $J$ steps at the true solution $Y$.
 
 \subsubsection*{Improvement of a family of supersolutions}

  At the end of Section \ref{sec improvement subsolutions} we explained how to improve a given family of subsolutions. The same idea can be applied here in order to simultaneously improve
  a family of supersolutions $(\bar{Y}^{\{l\}})_{l \in I}$, where $I$ is a finite index set. We now consider the predictable, 
  $I$-valued process 
 \[   
    l_*(j) = \inf \left\{ l \in K(j) \ \left| \ \forall \kappa \in K(j) \ F_{j-1} \left(E_{j-1} \left[\beta_j \bar{Y}_j^{\{l\}}\right] \right) 
	    \leq F_{j-1} \left(E_{j-1} \left[\beta_j \bar{Y}_j^{\{\kappa\}}\right] \right) \right. \right\}
  \]
  for every $j=1,...,J$, where $K(j)$ is again a nondecreasing family of subsets of $I$.  Then, the process $\bar{Y}$ defined by
  \[
   \bar{Y}_j= \bar{Y}_j^{\{l_*(j)\}} {\bf 1}_{\{j>0\}}+  F_0 \left(E_0 \left[ \beta_1 \bar{Y}_1 \right] \right) {\bf 1}_{\{j=0\}}
  \]
  is, by similar arguments as in Section \ref{sec improvement subsolutions}, a supersolution to \eqref{dynprog}. Thus, by Theorem \ref{improvement supersolution}, 
  \[
     E_j \left[\theta_j^{up} (\bar{M})\right] \leq F_j \left(E_j \left[\beta_{j+1} \bar{Y}_{j+1} \right] \right) 
	  = \min_{l \in K(j+1)} F_j \left( E_j \left[\beta_{j+1} \bar{Y}_{j+1}^{\{l\}} \right] \right) \leq \min_{l \in K(j+1)} \bar{Y}_j^{\{l\}}
  \]
  $P$-a.s. for every $j=0,...,J-1$, where $\bar{M}$ denotes the Doob martingale of $\beta \bar{Y}$. Hence, in the case $K(j)=I$ for $j=1,...,J$, improving $\bar{Y}$ results again in a
  simultaneous improvement of all supersolutions $(\bar{Y}^{\{l\}})_{l \in I}$.

\section{Implementation}
 \label{sec Implementation}

In this section, we explain how to implement algorithms based on the iterative improvement approaches of Sections \ref{sec improvement subsolutions} and \ref{sec improvement supersolutions}.
In order to transform these results into implementable algorithms, one needs to construct a sub- and a supersolution as input. Moreover, the conditional expectations which appear in the
 iterative constructions of the controls in \eqref{iteration subsolution control} and the Doob martingales 
in \eqref{iteration supersolution martingale} must be approximated numerically. For the numerical approximation of the conditional expectation within the iterative improvement we apply, as in \cite{KS}, a plain Monte Carlo 
implementation. In contrast to a naive plain Monte Carlo implementation of the dynamic programming equation \eqref{dynprog} (which leads to infeasible $J$ nested layers of simulation), the number of layers 
of simulation  in the iterative improvement algorithm depends on the number of iteration steps which are performed. As we shall demonstrate in the numerical examples, two improvement steps are  feasible, when the input super- and subsolutions are available in closed form. Therefore, we focus on the construction of closed-form inputs in Section 
\ref{sec input approximation}, before we explain the somewhat standard 
nested simulation approach for the iterative improvement in Section \ref{sec improvement approach}. 

As a first step, however, we specialize to the following Markovian framework: We assume that $(B_j)_{j=0,\ldots,J}$ is an $\DR^{\mathcal{D}}$-dimensional adapted process (with $\mathcal{D}\geq D$),
such that the first $D$ components of $B_j$ are given by $\beta_{j}$ and $B_{j}$ is independent of $\F_{j-1}$, for every 
$j=1,\ldots, J$.  $X$ is supposed to be an $\DR^N$-valued Markovian process of the form
\begin{align}
  X_{j} = h_{j}(X_{j-1},B_{j}), \quad j=1,...,J,
  \label{markovian process}
 \end{align} 
for  measurable functions $h_j:\DR^N \times \DR^\mathcal{D} \rightarrow \DR^N$, starting at $X_0=x_0 \in \DR^N$.  
 This forward equation for the state process $X$ could arise, e.g., as a time 
 discretization of a stochastic differential equation. Moreover, for the generator $F_j$ of the dynamic program \eqref{dynprog} we assume existence of measurable functions $f_j: \DR^N \times \DR^D \rightarrow \DR$ satisfying $F_j (\cdot) = f_j (X_j, \cdot)$,
 i.e., $F_j$ depends on $\omega$ only through the Markovian process $X$. Then, we consider a Markovian version of the dynamic program \eqref{dynprog} in the form
 \begin{align}
   Y_j &= f_j (X_j, E_j[\beta_{j+1} Y_{j+1}]),\quad j=0,\ldots, J-1, \quad  Y_J = g(X_J),
  \label{dynprog markovian}
 \end{align}
 where $g: \DR^N \rightarrow \DR$ is measurable. In this framework, $Y_j$ is a deterministic function of $X_j$ (and, in particular, $Y_0$ is a constant). In view of \eqref{markovian process}, we obtain, for every $j=1,\ldots,J$, a measurable 
function $y_j:  \DR^N \times \DR^\mathcal{D} \rightarrow \DR$ such that $Y_j=y_j(X_{j-1},B_j)$. Denoting by $P_{B_j}$ the law of $B_j$, we can, thus, write  
$E_j [\beta_{j+1} Y_{j+1}] = z_{j} (X_j)$ with 
\begin{align*}
z_j(x)=\left(\int_{\DR^\mathcal{D}}  b_1\, y_{j+1}(x,b) \, P_{B_{j+1}}(db),\ldots, \int_{\DR^\mathcal{D}}  b_D\, y_{j+1}(x,b) \, P_{B_{j+1}}(db)\right)^\top.
\end{align*}

\subsection{Computation of the input sub- and supersolution}
\label{sec input approximation}

For the construction of the input sub- and supersolutions, we first approximate $y_j$ by a linear combination of a given set of basis functions $\eta_j^1,...,\eta_j^K: \DR^N \times \DR^\mathcal{D} \rightarrow \DR$, i.e.,
\begin{align}\label{eq:input_approximation}
 \tilde y_j(x,b)=\sum_{k=1}^K a^k_j  \eta^k_j(x,b), \quad  j=1,\ldots,J.
\end{align}
 We consider two different ways to compute the $\F_0$-measurable coefficients $a^k_j$, a variant of least-squares 
Monte Carlo, which picks up some ideas of the regression later approach of \cite{GlYu04} in the optimal stopping literature, 
and a direct martingale minimization approach which builds on the works by \cite{Belo13} and \cite{DFM} for optimal stopping.  

As a key assumption on the basis functions, we impose that the expectations
\begin{align}\label{eq:basis}
 R_{j}^k (x) :=\left(\int_{\DR^\mathcal{D}}  b_1\, \eta^k_{j+1}(x,b) \, P_{B_{j+1}}(db),\ldots, \int_{\DR^\mathcal{D}}  b_D\, \eta^k_{j+1}(x,b) \, P_{B_{j+1}}(db)\right)^\top,
\end{align}
$x\in \DR^N$, are available in closed form (or, can be computed numerically up to a `negligible' error), cp. Remark \ref{modifMB} below. Defining $\tilde Y_j=\tilde y_j(X_{j-1},B_j)$ as an 
approximation to $Y_j$, we, thus, observe that $E_j[\beta_{j+1}\tilde Y_{j+1}]=\sum_k a^k_{j+1}  R_{j}^k (X_j)$ is given in closed form as well.
Based on the input approximation $\tilde y$, we can now derive first approximations of the optimal control and of the Doob martingale $M^*$, 
from which input sub- and supersolutions can be obtained via the pathwise recursions \eqref{recursion low} and \eqref{recursion up}. To compute such a control $\tilde{r}$, we 
solve \eqref{optimal control} with $Y_j$ replaced by 
$\tilde Y_j$, i.e. the process $\tilde{r}$ is given by
\begin{align}
 \tilde{r}_j^{\top} E_j [\beta_{j+1} \tilde Y_{j+1}] - f_j^{\#} (\tilde{r}_j) = f_j (X_j, E_j [\beta_{j+1} \tilde Y_{j+1}]), 
 \label{approximate optimal control}
\end{align}
for every $j=0,...,J-1$, and belongs to $\mathcal{A}_0$ thanks to Lemma A.1 in \cite{BGS}. Here, the convex conjugate can, of course, be approximated numerically as well. 
For the first approximation of the Doob martingale $M^*$, we, again, just replace the true solution $Y$ by its approximation $\tilde Y$. Hence, a first approximation $\tilde{M}$ of $M^*$ is given by 
\begin{align}\label{approximate Doob}
 \tilde{M}_j = \sum_{i=0}^{j-1} \beta_{i+1} \tilde{Y}_{i+1} - E_i [\beta_{i+1} \tilde{Y}_{i+1}], \quad j=0,...,J.
\end{align}
Plugging $\tilde{r}$ and $\tilde{M}$ into the recursions \eqref{recursion low} and \eqref{recursion up} for $\theta^{low}$ and $\theta^{up}$, we obtain the input sub- and supersolutions 
$Y_j^{(low,0)}=E_j[\theta^{low}_j(\tilde r,\tilde M)]$ and  $Y_j^{(up,0)}=E_j[\theta^{up}_j(\tilde M)]$.
 
We can now sample  $\Lambda^{out}$ independent copies $(B_j(\lambda^{out}), j=1,\ldots,J))_{\lambda^{out}=1,...,\Lambda^{out}}$ of $B$, to which we refer as `outer' paths. 
Then, we can compute the pathwise recursions for $\theta^{low}(\tilde r,\tilde M)$ and $\theta^{up}(\tilde M)$ along each of these outer paths and denote them by 
$\theta_j^{(low,0)}(\lambda^{out})$ and $\theta_j^{(up,0)}(\lambda^{out})$, $j=0,\ldots,J$, respectively. 
 Applying the plain Monte Carlo estimator
 \begin{align}
   \hat{Y}_0^{(up,0)} := \frac{1}{\Lambda^{out}} \sum_{\lambda^{out}=1}^{\Lambda^{out}} \theta_0^{(up,0)}(\lambda^{out})     \label{estimator Y0}
 \end{align}
for $E_0[\theta^{up}_0(\tilde M)]$ and the associated empirical standard deviation, one can compute an (asymptotic) confidence interval for $E_0[\theta^{up}_0(\tilde M)]$ and 
thus an upper confidence bound on $Y_0\leq E_0[\theta^{up}_0(\tilde M)] $, 
 see Section 1.1.3 of \cite{glasserman2003monte}. Analogously, from $(\theta_0^{(low,0)}(\lambda^{out}))_{\lambda^{out}=1,\ldots,\Lambda^{out}}$, a lower confidence bound can be constructed, 
and, combining both bounds, we end up with an asymptotic confidence interval. We emphasize that $Y_0$ is a deterministic real number, but the construction of the confidence interval 
is conditional on any set of sample paths which might be used to pre-compute the coefficients $a_j$ in  \eqref{eq:input_approximation} and which we think of as being included in $\F_0$.

 When such a confidence interval is not yet sufficiently tight for the application under consideration, one can run the 
iterative improvement algorithm described in Section \ref{sec improvement approach} below. We shall first, however, discuss two ways to obtain the coefficients for the input approximation \eqref{eq:input_approximation}.

 \subsubsection*{Least-squares Monte Carlo approach}
The idea of least-squares Monte Carlo (LSMC) is to approximate the conditional expectation in \eqref{dynprog markovian} by an orthogonal projection onto a set of basis functions via regression, i.e,
one computes
 \begin{align*}
   \tilde Y_j &= f_j (X_j, \mathcal{P}_j[\beta_{j+1} \tilde Y_{j+1}]),\quad j=0,\ldots, J-1, \quad  \tilde Y_J = g(X_J),
 \end{align*}
as an approximation to $Y$, where $\mathcal{P}_j$ denotes the empirical regression (given a set of sample paths) on a pre-specified basis. Note that one actually has to calculate $D$ empirical regressions in each time step, since the stochastic weight $\beta$ is
$\mathbb{R}^D$-valued, and that the expression $\beta_{j+1} \tilde Y_{j+1}$ may suffer from a large variance, e.g., in the BSDE case \eqref{eq:BSDE}, where the variance of 
the Malliavin Monte Carlo weights $\beta$ for the first space derivative explodes as the time 
discretization becomes finer and finer. With our standing assumption \eqref{eq:basis} on the basis functions we can, instead, implement the following single-regression variant of least-squares 
Monte Carlo:
\begin{align*}
   \tilde Y_j &= \mathcal{P}_j\left[f_j (X_j, E_j[\beta_{j+1} \tilde Y_{j+1}])\right],\quad j=0,\ldots, J-1, \quad  \tilde Y_J = \mathcal{P}_J[g(X_J)],
 \end{align*}
as (inductively) $\tilde Y_{j+1}$ is a linear combination of $(\eta_{j+1}^k(X_j,B_{j+1}))_{k=1,\ldots,K}$ and, thus, the conditional expectation inside $f_j$ is available in closed form. This
idea to `regress later' originates in \cite{GlYu04} for optimal stopping and was extended to the time discretization of BSDEs in \cite{BS12}, where a tremendous variance reduction effect is observed 
in the numerical examples.

To be more formal, we assume that $\Lambda^{reg}$ independent copies of $B$ to which we refer as 'regression paths' are given. The trajectories of $\beta$ and of the Markovian process $X$
along the $\lambda$th regression path are denoted by $\beta(\lambda)$ and $X(\lambda)$, $\lambda=1,\ldots,\Lambda^{reg}$. 
For the initialization of our algorithm, we require an approximation of the terminal condition $g_J(X_J)=y_J(X_{J-1},B_J)$ in terms of basis functions. Applying a standard regression approach, we compute 
$\DR^K$-valued coefficients $a_J = (a_J^1,...,a_J^K)$ via
\[
 a_J = \argmin_{a \in \DR^K} \frac{1}{\Lambda^{reg}} \sum_{\lambda=1}^{\Lambda^{reg}} \left( g (X_J(\lambda)) - \sum_{k=1}^K a^k \eta_J^k (X_{J-1}(\lambda),B_J(\lambda)) \right)^2
\]
and obtain $\tilde{y}_J(x,b)=\sum_{k=1}^K a_{J}^k \eta_{J}^k (x,b)$  as an approximation of $y_J(x,b)=g(h(x,b))$.
Now, assume that an approximation $\tilde{y}_{j+1}(x,b)$ in terms of the basis functions has already been computed, i.e.
$\tilde{y}_{j+1} (x,b) = \sum_{k=1}^K a_{j+1}^k \eta_{j+1}^k (x,b)$,
with $\DR^K$-valued coefficients $a_{j+1}=(a_{j+1}^1,...,a_{j+1}^K)$. Then, by \eqref{eq:basis},
\[
f_j \left(X_j, E_j \left[\beta_{j+1}  \sum_{k=1}^K a_{j+1}^k \eta_{j+1}^k (X_j,B_{j+1}) \right] \right)=
f_j \left(X_j, \sum_{k=1}^K a^k_{j+1}  R_{j}^k (X_j) \right),
\]
(where we, of course, formally, perform an initial enlargement of the filtration by the regression paths, which are assumed to be independent of $(X,\beta)$). Projecting the right-hand side empirically 
on the basis functions $(\eta^1_j,\ldots,\eta^K_j)$ leads to
\begin{eqnarray*}
 a_j &=& \argmin_{a \in \DR^K} \frac{1}{\Lambda^{reg}} \sum_{\lambda=1}^{\Lambda^{reg}} \Bigg( f_j \big(X_j(\lambda), \sum_{k=1}^K a^k_{j+1}  R_{j}^k (X_j(\lambda)) \big)   \left.- \sum_{k=1}^K a^k \eta_j^k (X_{j-1}(\lambda),B_j(\lambda)) \right)^2,
\end{eqnarray*}
and, thus, we obtain $\tilde y_j(x,b)= \sum_{k=1}^K a_{j}^k \eta_{j}^k (x,b)$ as an approximation to $y_j$. 

\begin{rem}\label{modifMB}
In contrast to \cite{GlYu04} and \cite{BS12}, we merely require in \eqref{eq:basis} that conditional expectations are available explicitly one step ahead, 
while \cite{GlYu04} and \cite{BS12} both additionally assume that the basis functions form martingales, i.e., $E_j[\eta^k_{j+1}(X_{j}, B_{j+1})]=\eta^k_j(X_{j-1}, B_j)$. 
 One can exploit the additional flexibility in the following way: 
Suppose that each basis function $\eta_j^k$ can be written in the product form $\eta_j^k(x,b)=\eta_j^{k,1}(x)\eta_j^{k,2}(h_j(x,b))$ and assume that $E [ \beta_j \eta_j^{k,2} (h_j(x, B_j))]$ is 
available in closed form. Then, the expression in \eqref{eq:basis}
is also available in closed form, as required. In particular, while the choice of $\eta_j^{k,2}$ is restricted to functions where explicit computations are possible, we are completely flexible in capturing a 
more complex dependence on the process $X$ through the factor $\eta_j^{k,1}(x)$. In the numerical example of Section \ref{sec funding}, we illustrate such a choice of basis functions. 
\end{rem}

 \subsubsection*{Martingale minimization approach}

It has been observed in the context of optimal stopping and in the BSDE examples in \cite{BSZ} and \cite{BGS} that the construction of tight supersolutions can be significantly more 
difficult than the construction of tight subsolutions. The idea of the martingale minimization approach is thus to compute the coefficients in \eqref{eq:input_approximation} in such a way 
that the upper bound implied by the input supersolution is minimized.  In the context of optimal stopping, similar ideas have been developed in \cite{DFM} and \cite{Belo13}. In contrast to 
 least-squares Monte Carlo, the optimization is now global 
and so the coefficients in \eqref{eq:input_approximation} do not depend on the time index $j$.

In view of \eqref{eq:basis} each basis function defines a martingale via $M_0^{\{k\}}=0$ and
\[
 M_j^{\{k\}} - M_{j-1}^{\{k\}} = \beta_j \eta_j^k (X_{j-1},B_j) - R_{j-1}^k (X_{j-1}), \quad k=1,...,K.
\]
Writing,
\begin{equation}
  M_j^{a} = \sum_{k=1}^K a^{k} M_j^{\{k\}}, \quad j=0,...,J,
  \label{martingale minimization}
 \end{equation}
where $a=(a^{1},...,a^{K}) \in \DR^K$, we wish to choose a coefficient vector $a^*$, for which $E_0[\theta^{up}_0(M^a)]$ becomes minimal. Taking the pathwise optimality of the optimal martingale $M^*$ in 
\eqref{eq:optimal_martingale}
into account and following the approach analyzed in \cite{Belo13} for optimal stopping, we add a standard deviation penalty to this minimization problem.   
To make the approach implementable, the expectation and standard deviation need to be replaced by empirical estimators over sample paths. To this end, 
we sample $\Lambda^{mini}$ independent copies of $B$ (which we refer to as `minimization paths') and denote the evaluation of $\beta$ and $X$ along the $\lambda$th minimization path by
$\beta(\lambda)$ and $X(\lambda)$.
We then solve for
\begin{align}
 a^* = \argmin_{a \in \DR^K} \left(\hat{E}_0[\theta_0^{up}(M^a)] + \gamma \sqrt{\frac{1}{\Lambda^{mini}-1} \sum_{\lambda=1}^{\Lambda^{mini}} \left(\theta_0^{up}(M^a; \lambda) - \hat{E}[\theta_0^{up}(M^a)] \right)^2 }\right),
 \label{minimization empirical}
\end{align}
where $\gamma\geq 0$ is fixed,
\[
 \hat{E}_0[\theta_0^{up}(M^a)] = \frac{1}{\Lambda^{mini}} \sum_{\lambda=1}^{\Lambda^{mini}} \theta_0^{up} (M^a; \lambda),
\]
and $\theta^{up}(M^a; \lambda)$ is sampled according to \eqref{recursion up} along the $\lambda$th minimization path. 
In our numerical examples, we use the {\sc Matlab} implementation of the Nelder-Mead simplex algorithm to search for $a^*$.
An approximation to $y_j$ as in \eqref{eq:input_approximation} 
is then given by
$$
\tilde{y}_j (x,b) = \sum_{k=1}^K a^{k,*} \eta_j^k (x,b), \quad j=1,...,J.
$$

\begin{rem}\label{rem:training}
 The minimization approach requires the choice of the parameter $\gamma$. In our numerical results presented in Section \ref{sec funding}, we apply a ``training and testing'' approach to tune this 
 parameter. To this end, we choose a set $\{\gamma_1,...,\gamma_L\}$, $L \in \mathbb{N}$, of parameters. For each $\gamma_l$, $l=1,...,L$, we compute a vector of coefficients 
 $a^*_{\gamma_l} \in \DR^K$ according to \eqref{minimization empirical} along the minimization paths $\Lambda^{mini}$. If vectors $a^*_{\gamma_1},...,a^*_{\gamma_L}$ are computed, we sample
 a new set of $\Lambda^{test}$ test paths (independent copies of $B$ which are also independent of the minimization paths). The parameter $\gamma$ is obtained by taking the  
 $\gamma_l$ such that $a^*_{\gamma_l}$  minimizes the expression in  brackets on the right hand side of \eqref{minimization empirical} along the test paths over the set $\{ a^*_{\gamma_1},...,a^*_{\gamma_L}\}$.
We note that in our experience the method's practical performance 
 is not particularly sensitive to the choice of $\gamma$.
\end{rem}

 \subsection{Iterative improvement algorithm}
 \label{sec improvement approach}

We now assume that we are given input super- and subsolutions of the form  $Y^{(up,0)}_j=E_j[\theta^{up}(\tilde M)]$ and $Y^{(low,0)}_j=E_j[\theta^{low}(\tilde r,\tilde M)]$
such that the control $\tilde r \in \mathcal{A}_0$ and the martingale $\tilde M \in \mathcal{M}_D$ can be evaluated in closed form along a given path $B$, cp. the constructions 
in Section \ref{sec input approximation}. In order to compute the first iteration $\theta^{(low,1)}$ in \eqref{iteration subsolution} and $\theta^{(up,1)}$ in \eqref{iteration supersolution} we need to approximate the conditional 
expectations $E_j[\beta_{j+1} \theta^{low}_{j+1}(\tilde r,\tilde M)]$,  $E_j[\beta_{j+1} \theta^{up}_{j+1}(\tilde M)]$, and $E_{j+1}[ \theta^{up}_{j+1}(\tilde M)]$. In the following, we focus on 
the supersolution case, but note that the subsolution case is analogous.

In our plain Monte Carlo implementation, we first sample $\Lambda^{out}$ independent copies $B(\lambda^{out})$, $\lambda^{out}=1,\ldots \Lambda^{out}$, of $B$. Moreover, for every 
time step $j$ and outer path $B(\lambda^{out})$, we generate a new sample of independent copies $(B_i(\lambda^{mid},j))_{i\geq j+1}$, $\lambda^{mid}=1,\ldots \Lambda^{mid}$, of $(B_i)_{i\geq j+1}$.
We denote by $B(\lambda^{out},\lambda^{mid}, j)$ the path given by $(B_1(\lambda^{out}), \ldots, B_j(\lambda^{out}), B_{j+1}(\lambda^{mid},j), \ldots, B_{J}(\lambda^{mid},j))$, which switches 
from a given outer path to the corresponding middle path at time $j+1$. Similarly to the notation introduced before, we write $\beta(\lambda^{out},\lambda^{mid}, j)$ and 
$\theta^{(up,0)}(\lambda^{out},\lambda^{mid}, j)$ for the trajectories of $\beta$ and   $\theta^{up}(\tilde M)$ along the path $B(\lambda^{out},\lambda^{mid}, j)$. Along each outer path, we approximate 
the martingale $M^{(1)}$ in \eqref{iteration supersolution martingale} with increment 
 \[
   M_{j+1}^{(1)}-M_j^{(1)} = \beta_{j+1} E_{j+1} [\theta_{j+1}^{(up,0)}] - E_j [\beta_{j+1} \theta_{j+1}^{(up,0)}], \quad j=0,...,J-1,
 \]
by the plain Monte Carlo estimator
\[
  \tilde{M}_{j+1}^{(1)}(\lambda^{out}) - \tilde{M}_j^{(1)}(\lambda^{out}) = \beta_{j+1}(\lambda^{out}) \hat{E}_{j+1} [\theta_{j+1}^{(up,0)}](\lambda^{out}) - \hat{E}_j [\beta_{j+1}\theta_{j+1}^{(up,0)}](\lambda^{out})
 \]
where
 \begin{align}
   \hat{E}_j [\theta_j^{(up,0)}](\lambda^{out}) &:= \frac{1}{\Lambda^{mid}} \sum_{\lambda^{mid}=1}^{\Lambda^{mid}} \theta_j^{(up,0)} (\lambda^{out},\lambda^{mid},j) \nonumber \\ 
   \hat{E}_j [\beta_{j+1}\theta_{j+1}^{(up,0)}](\lambda^{out}) &:= \frac{1}{\Lambda^{mid}} \sum_{\lambda^{mid}=1}^{\Lambda^{mid}} \beta_{j+1} (\lambda^{out},\lambda^{mid},j) \theta_{j+1}^{(up,0)} (\lambda^{out},\lambda^{mid},j). \label{estimator cond exp} 
 \end{align}
Standard calculations show that $\tilde M^{(1)}$ is also a martingale when the filtration is suitably enlarged by the middle paths. We now write $\theta^{(up,1)}(\lambda^{out})$ for the realization 
of $\theta^{up}(\tilde M^{(1)})$ along the $\lambda^{out}$th outer path. Proceeding as in \eqref{estimator Y0} ff., we can compute a new upper confidence bound for $Y_0$
based on $(\theta^{(up,1)}(\lambda^{out}))_{\lambda^{out}=1,\ldots, \Lambda^{out}}$. Since $\tilde M^{(1)}$ converges 
to $M^{(1)}$ (along each outer path) as the number of middle paths converges to infinity, and since $E_0[\theta^{up}(M^{(1)})]\leq E_0[ \theta^{up}(\tilde M)]$ by Theorem \ref{improvement supersolution},
the corresponding upper bound is typically tighter than the one constructed from $(\theta^{(up,0)}(\lambda^{out}))_{\lambda^{out}=1,\ldots, \Lambda^{out}}$, when the number of middle paths is sufficiently large.

If one wishes to compute a second iteration step (e.g., because the once improved confidence interval is still not tight enough), one can repeat this procedure with the only difference 
that we cannot assume  the input martingale (which now is $\tilde M^{(1)}$) to be available in closed form along a given path. Its evaluation actually requires one layer of nested 
simulation as described above. However, in the next iteration step $\tilde M^{(1)}$ must be evaluated along middle paths and not along outer paths, and so the sampling of a third layer 
of $\Lambda^{in}$ `inner paths' is required. We do not get into any more details of the straightforward implementation, but note that, analogously, a third layer of simulation must already be 
sampled
in the first iteration step, when the input martingale $\tilde M$ is not available in closed form (e.g., when we drop assumption \eqref{eq:basis}). 

In order to reduce the number of middle paths (in the first iteration step) and inner paths (in the second iteration step), we suggest to apply control variates in the plain Monte Carlo estimation 
\eqref{estimator cond exp}  of 
the martingale increments based on the closed form expression for $E_j[ \theta^{up}_{j+1}(\tilde M)]$ and $E_j[\beta_{j+1} \theta^{up}_{j+1}(\tilde M)]$. In our actual implementation, we proceed as 
described at the end of  Section 4.1 in \cite{BSZ}.

In principle, the algorithm can be further iterated, but each iteration step adds an additional layer of simulations. So, for practical reasons, we recommend not to run the algorithm
with more than three layers of simulations, but rather to put more effort into the construction of the input approximations, when the confidence interval is still not tight enough. In our numerical 
test example below, very satisfactory 95\% confidence bounds can be obtained with two iteration steps, even when the input approximation $\tilde y$ in \eqref{eq:input_approximation} is pre-computed by 
the martingale minimization approach with a single constant basis function.

\section{Numerical example}
\label{sec funding}
 
 In this section, we apply our approach to the problem of pricing a European option under funding constraints, i.e., under different interest rates for borrowing and lending. In the finance 
 literature, this problem goes back to \cite{bergman1995option}. \cite{laurent2014overview} emphasizes the relevance of such models in the light of the recent 
 financial crisis.  The model is also prominent example in the literature on backward stochastic differential equations starting with  \cite{KPQ} and a well-established numerical test case 
 \cite{gobet2005regression, LGW, BS12, BSZ}. We begin by setting up the problem and explaining how it fits into our 
 framework. Then, we present our numerical results. For the computation of input approximations we present different approaches, which incorporate a priori knowledge about the problem to a varying 
 extent. The upper and lower bounds as well as the corresponding improvements are computed relying on the methodology of Section \ref{sec Implementation}, including the use of control variates.

 \subsubsection*{Pricing under funding risk}
 
 Let $0=t_0 < t_1 <...<t_J=T$ be an equidistant discretization of the interval $[0,T]$. There are two riskless interest rates $R^l < R^b \in \DR$ for lending
 respectively borrowing and $N$ risky assets given by geometric Brownian motions $X_1,...,X_N$ with dynamics 
 \[
  X_{n,j} = x_{n,0} \exp \left\{ \left( \mu - \frac{1}{2} \sum_{l=1}^N \sigma_{n,l}^2 \right) t_j + \sum_{l=1}^N \sigma_{n,l} W_{l,t_j} \right\}, \quad n=1,...,N,
 \]
 at $t_j$, for $j=0,...,J$. Here, $x_{n,0},\mu \in \DR$, $\sigma$ is an invertible $N \times N$-matrix with entries in $\DR$ and $W_1,...,W_N$ are independent Brownian motions. 
 We consider the problem of pricing a European
 option on the assets $X_1,...,X_N$ with maturity $T$ and payoff $g(X_{1,J},...,X_{N,J})$. When $g$ satisfies a polynomial growth condition, the option payoff belongs to $\lp_J (\DR)$. Applying the 
 discretization scheme proposed in \cite{FTW} to equation (1.11) in \cite{KPQ}, the value $Y$ of the option on the time grid $\{t_0,...,t_J \}$ is given by
 \begin{align}
    \label{dyn prog funding} 
  Y_j &= (1-R^l \Delta ) E_j [Y_{j+1}] - (\mu-R^l) Z_j^{\top} \sigma^{-1} \textbf{1} \Delta + (R^b-R^l) \Delta (E_j [Y_{j+1}] - Z_j^{\top} \sigma^{-1} \textbf{1})_-, 
  \end{align}
with terminal condition $Y_J = g(X_{1,J},...,X_{N,J})$.
 Here, $\Delta := t_j - t_{j-1}$ for $j=1,...,J$, $\textbf{1} \in \DR^N$ is a vector consisting of ones, and $(x)_- := \max \{-x,0\}$ for $x \in \DR$. 
 Moreover, the random vector $Z_j$ is given by 
 \[
  Z_{n,j} := E_j \left[ \frac{\Delta W_{n,j+1}}{\Delta} Y_{j+1} \right], \quad n=1,...,N,
 \]
 where $E_j [\cdot]$ denotes the conditional expectation with respect to the information generated by the multivariate Brownian motion up to time $t_j$, and $\Delta W_{n,j+1} := W_{n,t_{j+1}} - W_{n,t_j}$. The term $Z_j^{\top} \sigma^{-1} \textbf{1}$ in 
 \eqref{dyn prog funding} represents the overall position in the risky assets in the hedging portfolio at time $t_j$. Therefore, $E_j [Y_{j+1}] - Z_j^{\top} \sigma^{-1} \textbf{1}$ is an approximation
 of the position in the bank account at time $t_j$. The sign of this expression determines which interest rate is applicable.
  By taking the function $F_j: \DR^D \rightarrow \DR$, $D=N+1$, given by
 \[
  F_j (z) = (1-R^l \Delta) z_0 - (\mu-R^l) z_{(-0)}^{\top} \sigma^{-1} \textbf{1} \Delta + (R^b-R^l) \Delta \left(z_0 - z_{(-0)}^{\top} \sigma^{-1} \textbf{1} \right)_-, 
 \]
 where $z_{(-0)} := (z_1,...,z_N)$, and setting
 \[
  B_{j+1} = \beta_{j+1} = \left( 1, \frac{p_C \left(\Delta W_{1,j+1} \right)}{\Delta}, ..., \frac{p_C \left(\Delta W_{N,j+1} \right)}{\Delta} \right)^{\top}, \quad j=0,...,J-1,
 \]
 we observe that the recursion \eqref{dyn prog funding} fits into our framework. Here, $p_C$ denotes a truncation function, i.e. $p_C (x) = -C \vee x \wedge C$ for $C \in \DR_+$. 
 Notice that the truncation of $W_{t_{j+1}}-W_{t_j} \sim \mathcal{N}(0,\Delta)$ at $C$ becomes arbitrarily mild as $\Delta$ gets small.
 Truncation is required, because the increments of the Brownian motions are unbounded and, thus, the monotonicity assumption \eqref{monotonicity} might be violated. 
 A sufficient condition for \eqref{monotonicity} to hold is then
\begin{align}
\Delta \cdot \max\{|R^l|,|R^b|\}+C\cdot \max\{|\mu-R^l|,|\mu-R^b|\} \cdot \sum_{d=1}^D\left|\sum_{l=1}^D (\sigma^{-1})_{d,l}\right| \leq 1,
\label{truncation}
\end{align}
see \cite{LGW} for an analysis of this type of truncation error.

Recall that our algorithm requires the convex conjugate $F_j^\#$ of $F_j$ when running the pathwise recursion formula for  $\theta^{low}$ and it also requires to solve for the optimality condition \eqref{iteration subsolution control} in the iteration 
for the lower bounds. As $F$ is piecewise linear, the convex conjugate is straightforward to compute and equals $F^\#_j \equiv 0$ on its effective domain. Moreover, with the  function 
$u: \DR \rightarrow \DR^{N+1}$ given by
  \[
   u^{(0)}(s) = (1-s \Delta) \quad \textnormal{ and } \quad u^{(n)}(s) = -(\mu-s)\Delta \sum_{l=1}^N \left(\sigma^{-1}\right)_{n,l},\quad n=1,\ldots,N,
  \]
the effective domain of the convex conjugate is  $D_{F^{\#}}^{(j,\omega)}=\{u(R)|R\in [R^l,R^b]\}$. Finally, for every $z=(z_0,\ldots,z_N)\in \DR^{N+1}$, 
 \begin{align*}
   r(z) = 
   \begin{cases}
    u(R^l), & z_0 \geq (z_1,...,z_N) \sigma^{-1} \textbf{1} \\
    u(R^b), & z_0 < (z_1,...,z_N) \sigma^{-1} \textbf{1} 
   \end{cases}
  \end{align*}
solves $r(z)^\top z=F(z)$, compare \eqref{iteration subsolution control}.

\subsubsection*{Benchmark product}

For our numerical experiments, we consider the example discussed in \cite{BSZ}, but add a non-trivial correlation structure to the problem. This example is
 a multidimensional version of an example going back to \cite{gobet2005regression}. We compute upper and lower bounds on 
the price of a European call-spread option with strikes $K_1$ and $K_2$ on the maximum of five assets, i.e.,
\[
 g(x_1,...,x_5) = \left( \max_{d=1,...,5} x_d - K_1 \right)_+ - 2\left( \max_{d=1,...,5} x_d - K_2 \right)_+, \quad x \in \DR^5.
\]
The maturity $T$ is set to three months, i.e. $T=0.25$, and the strikes are $K_1=95$ and $K_2=115$. The interest rates $R^l$ and $R^b$
are $1\%$ and $6\%$. For the geometric Brownian motions $X_1,...,X_5$ we take $x_{d,0}=100$, $d=1,...,5$, as starting value and choose the drift $\mu$ to be $0.05$. In contrast 
to \cite{BSZ} we do not assume that $X_1,...,X_5$ are independent and consider the diffusion matrix $\sigma$ given by
\[
 \sigma = \tilde{\sigma} \cdot
 \begin{pmatrix}
  1 & 0 & 0 & 0 & 0 \\
  \rho & \sqrt{1-\rho^2} & 0 & 0 & 0 \\
  \rho & 0 & \sqrt{1-\rho^2} & 0 & 0 \\
  \rho & 0 & 0 & \sqrt{1-\rho^2} & 0 \\
  \rho & 0 & 0 & 0 & \sqrt{1-\rho^2} \\
 \end{pmatrix},
\]
where $\tilde{\sigma}=0.2$. In our numerical experiments below, the correlation parameter $\rho$ varies in the interval $[-0.3,0.3]$ and the time discretization $J$ takes values in $\{20, 30, 40\}$. 
With this choice of parameters, we observe that 
 \eqref{truncation} holds with $C= 0.77$ at the roughest time discretization level $J=20$. Truncating the Brownian 
increments with standard deviation $\sqrt{\Delta} \approx 0.112$ at 0.77 is the same as truncating a standard normal random variable at 6.88, corresponding to truncating a probability mass of 
$3\cdot10^{-12}$ in both tails.

\subsubsection*{Generic minimization algorithm}

For the construction of the input approximation, we first run the martingale minimization algorithm with the single and completely generic basis function  
$\eta_j^1 (x,b) := 1$, i.e., we initially approximate $Y_j$ by a constant and the $Z_{n,j}$ by zero, $n=1,\ldots,N$. Then, in 
 the minimization approach presented in Section \ref{sec input approximation} we have a single 6-dimensional martingale $M^{\{1\}}$ given by $\tilde{M}_{0,j+1}^{\{1\}} - \tilde{M}_{0,j}^{\{1\}} = 0$ 
and
\[
 \tilde{M}_{d,j+1}^{\{1\}} - \tilde{M}_{d,j}^{\{1\}} = \beta_{d,j+1} - E_j [\beta_{d,j+1}] = \frac{p_C \left(\Delta W_{d,j+1} \right)}{\Delta}
\]
for $d=1,...,5$. In order to compute the $\mathbb{R}$-valued coefficient $a^*$, and, hence, the constant approximation $\tilde y_j(x,b)=a^*$ to $y_j$, we implement the `training and testing' approach of 
Remark \ref{rem:training}  with $\Lambda^{mini} = \Lambda^{test} = 1000$ paths and $\{\gamma_1,...,\gamma_{21}\} = \{0, \ 0.025,..., 0.5 \}$. We 
find that $a^*$, as an approximation of $Y_0$, ranges between 16 and 17.5 for our different choices of $J$ and $\rho$, and as $a^*>0$, the 
input subsolution $Y^{(low,0)}$ 
is constructed from the constant control $u(R^l)$. 
For the computation of upper and lower bounds with up to two iterative improvements, we take $\Lambda^{out}=1000$ outer paths, $\Lambda^{mid}=200$ middle paths and $\Lambda^{in}=50$ inner paths. 
The resulting estimators for the upper and lower bounds
from the $k$-th improvement are denoted by $\hat{Y}_0^{(up,k),a}$ and $\hat{Y}_0^{(low,k),a}$. For comparison, we also state the upper bound estimator  $\hat{Y}_0^{(up,0),0}$ which is computed by choosing $a=0$,
i.e., by setting all martingale increments to zero. 

Table \ref{table generic} presents upper and lower bounds for two different choices of 
$\rho$, namely $\rho=0.3$ and $\rho=-0.3$.

\begin{table}[htbp]
 \centering
 \begin{tabular}{|c||c|c|c||c|c|c|}
 \hline  
 $\rho$ & \multicolumn{3}{|c||}{$0.3$} & \multicolumn{3}{|c|}{$-0.3$}\\
 \hline 
 $J$ & 20 & 30 & 40 & 20 & 30 & 40 \\
 \hline \vspace{-11pt} & & & & & &\\
 $\hat{Y}_0^{(up,0),0}$ & $\underset{(0.2243)}{18.9637}$ & $\underset{(0.2444)}{20.5682}$ & $\underset{(0.2720)}{22.1942}$ & $\underset{(0.2736)}{26.1759}$ & $\underset{(0.3217)}{29.5829}$ & $\underset{(0.3424)}{34.2500}$ \\
 \hline \vspace{-11pt} & & & & & &\\
 $\hat{Y}_0^{(up,0),a^*}$ & $\underset{(0.1405)}{14.4278}$ & $\underset{(0.1330)}{14.5533}$ & $\underset{(0.1343)}{14.7838}$ & $\underset{(0.1081)}{15.9557}$ & $\underset{(0.1025)}{16.0660}$ & $\underset{(0.0948)}{16.5631}$  \\
 \hline \vspace{-11pt} & & & & & &\\
 $\hat{Y}_0^{(up,1),a^*}$ & $\underset{(0.0129)}{13.1430}$ & $\underset{(0.0133)}{13.2626}$ & $\underset{(0.0154)}{13.3451}$ & $\underset{(0.0127)}{14.5063}$ & $\underset{(0.0148)}{14.6878}$ & $\underset{(0.0136)}{14.9476}$  \\
 \hline \vspace{-11pt} & & & & & &\\
 $\hat{Y}_0^{(up,2),a^*}$ & $\underset{(0.0127)}{13.0461}$ & $\underset{(0.0137)}{13.1088}$ & $\underset{(0.0139)}{13.0919}$ & $\underset{(0.0107)}{14.2047}$ & $\underset{(0.0108)}{14.2452}$ & $\underset{(0.0102)}{14.3340}$  \\
 \hline \vspace{-11pt} & & & & & &\\
 $\hat{Y}_0^{(low,0),a^*}$ & $\underset{(0.0231)}{12.6157}$ & $\underset{(0.0253)}{12.6281}$ & $\underset{(0.0307)}{12.5792}$ & $\underset{(0.0289)}{13.7919}$ & $\underset{(0.0366)}{13.7291}$ & $\underset{(0.0368)}{13.8283}$ \\
 \hline \vspace{-11pt} & & & & & &\\
 $\hat{Y}_0^{(low,1),a^*}$ & $\underset{(0.0139)}{12.9915}$ & $\underset{(0.0150)}{13.0063}$ & $\underset{(0.0184)}{12.9703}$ & $\underset{(0.0180)}{14.0492}$ & $\underset{(0.0227)}{14.0000}$ & $\underset{(0.0233)}{14.0498}$  \\
 \hline
 \end{tabular}
 \caption{Upper and lower bounds based on the generic minimization algorithm for different time discretizations. Standard deviations are given in brackets.}
 \label{table generic}
\end{table}
We first observe that the upper bound is very sensitive with respect to the input martingale. Even optimizing a very crude constant approximation for $Y$ has a huge impact, and, e.g., leads to
a half as large upper bound for $J=40$ time steps in the negative correlation case compared to the upper bound computed from the zero martingale $\hat{Y}_0^{(up,0),0}$. Nonetheless, the relative width
of the 95\% confidence interval based on the optimal constant approximation is still more than 20\% for 40 time steps in the positive correlation case and even larger in the negative correlation case. 
Improving 
upper and lower confidence bound once, shrinks the 95\% confidence interval to a quite acceptable relative width of less than 3.5\% in the positive correlation case, while a second iterative 
improvement of the upper bound leads to    
relative width of less than 1.5\%. The negative correlation apparently makes the problem harder to solve numerically. But, still, after two iteration steps for the upper bound and one iteration step for the lower 
bound we end up with a 95\% confidence interval of a relative width of less than 2.5\%.
We also observe a significant decrease in the empirical standard deviations of the upper bound estimators through the improvement steps,  
as expected since the martingales approach the pathwise optimal Doob martingale of $\beta Y$, cp. \eqref{eq:optimal_martingale}.
 
Taking into account that no problem-specific information was used to construct the above confidence intervals in a five-dimensional 
problem with non-smooth coefficients and non-trivial correlation 
structure, we believe that the numerical results  are rather striking. We note, however, that the second iteration step increases the computational costs by a factor of $\Lambda^{in}\cdot (J/3)$ (e.g.,
a factor of 667 in our setting for $J=40$ time steps)
compared to a single improvement step. Thus, we next explore to what extent the results can be improved by putting more effort into the construction of the input approximation.

\subsubsection*{Non-generic minimization and LSMC algorithms}

Following ideas of \cite{AB04} for the pricing of Bermudan options on the maximum of several assets, we now incorporate information about option prices 
on the largest and second-largest asset into the function basis.
To this end, we define the two adapted processes $d^{(1)}$ and $d^{(2)}$ by
\begin{align*}
 d_j^{(1)} &:= \inf \left\{ d \in \{1,...,5\} \left| X_{d,j} \geq X_{n,j} \ \forall n=1,...,5 \right. \right\} \\
 d_j^{(2)} &:= \inf \left\{ d \in \{1,...,5\}\setminus\{d_j^{(1)}\} \left| X_{d,j} \geq X_{n,j} \ \forall n \in \{1,...,5\}\setminus\{d_j^{(1)}\} \right. \right\} 
\end{align*}
for $j=0,...,J$. Hence, $d_j^{(1)}$ and $d_j^{(2)}$ indicate the largest respectively second-largest asset at time $t_j$. In particular, they can be viewed as functions of $X_j$. Based on this, we 
define the following functions which serve as a basis for our approximations of $Y$:
\begin{align*}
 \eta_j^1 (X_{j-1},X_j) &:= 1, \quad
 \eta_j^{\iota+1} (X_{j-1},X_j) := \sum_{d=1}^5 X_{d,j} 1_{\{d_{j-1}^{(\iota)} = d \}},\; \iota=1,2, \\
  \eta_j^{\iota+3} (X_{j-1},X_j) &:= \sum_{d=1}^5 E \left[ \left. \left( X_{d,J} - K_1 \right)_+ - 2\left( X_{d,J} - K_2 \right)_+  \right|X_{d,j}\right] 1_{\{d_{j-1}^{(\iota)} = d \}} ,\; \iota=1,2, \\
 \eta_j^6 (X_{j-1},X_j) &:= \sum_{d=1}^5 E \left[ \left. \left( X_{d,J} - K_2 \right)_+ \right|X_{d,j}\right] 1_{\{d_{j-1}^{(1)} = d \}}. 
\end{align*}
For $j=0$, we replace $1_{\{d_{j-1}^{(\iota)} = d \}}$ by $1_{\{d_j^{(\iota)} = d \}}$, $\iota = 1,2$. 
Here, we write, for simplicity and in slight abuse of notation, the basis functions as functions of $(X_{j-1},X_j)$ instead of $(X_{j-1},B_j)$. Note that, e.g., the fourth basis function represents 
the price of the corresponding call spread option at time $t_j$ on the asset which is the largest one at time $t_{j-1}$. Shifting the time index in the indicator by one time step (compared 
to the more intuitive function basis in \cite{AB04} which is based on the largest asset at time $t_j$) 
turned out to be inessential in this numerical example, but ensures that the
 'one-step' conditional expectations $R_{j-1}^k (X_{j-1})$ in \eqref{eq:basis} are available in closed form.
These are, essentially, Black-Scholes prices and Black-Scholes deltas of European options at time $t_{j-1}$ on the asset which is the (second) largest at time $t_{j-1}$.

With these basis functions, we construct input approximations as described in Section \ref{sec input approximation}. For the martingale minimization algorithm,
 we run as before $\Lambda^{mini}= \Lambda^{test}=1000$ paths and take the penalization parameter from the set  $\{\gamma_1,...,\gamma_{21}\} = \{0, \ 0.025, ..., 0.5 \}$.
The modified LSMC approach is applied with $\Lambda^{reg}=100.000$ regression paths.
Tables \ref{table non-generic minimization} and \ref{table non-generic regression} below display the corresponding 
upper and lower bound estimators as well as iterative improvements up to the second order, based on these two input
approximations. As before, we denote by $\hat{Y}_0^{(up,k)}$ and $\hat{Y}_0^{(low,k)}$ the upper respectively lower bound resulting from the $k$-th improvement.
\begin{table}[htbp]
 \centering
 \begin{tabular}{|c||c|c|c||c|c|c|}
 \hline  
 $\rho$ & \multicolumn{3}{|c||}{$0.3$} & \multicolumn{3}{|c|}{$-0.3$}\\
 \hline 
 $J$ & 20 & 30 & 40 & 20 & 30 & 40 \\
 \hline \vspace{-11pt} & & & & & &\\
 $\hat{Y}_0^{(up,0),mini}$ & $\underset{(0.0694)}{13.3465}$ & $\underset{(0.0738)}{13.3766}$ & $\underset{(0.0741)}{13.5420}$ & $\underset{(0.0658)}{14.7198}$ & $\underset{(0.0676)}{14.9629}$ & $\underset{(0.0634)}{15.0104}$  \\
 \hline \vspace{-11pt} & & & & & &\\
 $\hat{Y}_0^{(up,1),mini}$ & $\underset{(0.0065)}{13.0424}$ & $\underset{(0.0072)}{13.0595}$ & $\underset{(0.0076)}{13.0738}$ & $\underset{(0.0065)}{14.2064}$ & $\underset{(0.0064)}{14.2828}$ & $\underset{(0.0071)}{14.3567}$ \\
 \hline \vspace{-11pt} & & & & & &\\
 $\hat{Y}_0^{(up,2),mini}$ & $\underset{(0.0070)}{13.0423}$ & $\underset{(0.0070)}{13.0768}$ & $\underset{(0.0071)}{13.0751}$ & $\underset{(0.0060)}{14.1761}$ & $\underset{(0.0056)}{14.1939}$ & $\underset{(0.0057)}{14.2293}$  \\
 \hline \vspace{-11pt} & & & & & &\\
 $\hat{Y}_0^{(low,0),mini}$ & $\underset{(0.0076)}{12.9953}$ & $\underset{(0.0102)}{12.9737}$ & $\underset{(0.0098)}{12.9923}$ & $\underset{(0.0112)}{14.0466}$ & $\underset{(0.0093)}{14.0773}$ & $\underset{(0.0123)}{14.0754}$  \\
 \hline \vspace{-11pt} & & & & & &\\
 $\hat{Y}_0^{(low,1),mini}$ & $\underset{(0.0068)}{13.0167}$ & $\underset{(0.0076)}{13.0171}$ & $\underset{(0.0082)}{13.0101}$ & $\underset{(0.0075)}{14.0835}$ & $\underset{(0.0075)}{14.1011}$ & $\underset{(0.0093)}{14.1035}$ \\
 \hline
 \end{tabular}
 \caption{Upper and lower bounds based on the non-generic minimization algorithm for different time discretizations. Standard deviations are given in brackets.}
 \label{table non-generic minimization}
\end{table}

\begin{table}[htbp]
 \centering
 \begin{tabular}{|c||c|c|c||c|c|c|}
 \hline  
 $\rho$ & \multicolumn{3}{|c||}{$0.3$} & \multicolumn{3}{|c|}{$-0.3$}\\
 \hline 
 $J$ & 20 & 30 & 40 & 20 & 30 & 40 \\
 \hline \vspace{-11pt} & & & & & &\\
 $\hat{Y}_0^{(up,0),reg}$ & $\underset{(0.0654)}{13.2481}$ & $\underset{(0.0660)}{13.3234}$ & $\underset{(0.0694)}{13.3730}$ & $\underset{(0.0720)}{14.7348}$ & $\underset{(0.0726)}{14.9905}$ & $\underset{(0.0782)}{14.9565}$ \\
 \hline \vspace{-11pt} & & & & & &\\
 $\hat{Y}_0^{(up,1),reg}$ & $\underset{(0.0061)}{13.0439}$ & $\underset{(0.0057)}{13.0479}$ & $\underset{(0.0056)}{13.0675}$ & $\underset{(0.0060)}{14.2022}$ & $\underset{(0.0063)}{14.2704}$ & $\underset{(0.0064)}{14.3315}$  \\
 \hline \vspace{-11pt} & & & & & &\\
 $\hat{Y}_0^{(up,2),reg}$ & $\underset{(0.0064)}{13.0503}$ & $\underset{(0.0065)}{13.0681}$ & $\underset{(0.0067)}{13.0857}$ & $\underset{(0.0057)}{14.1840}$ & $\underset{(0.0056)}{14.2207}$ & $\underset{(0.0059)}{14.2351}$  \\
  \hline \vspace{-11pt} & & & & & &\\
 $\hat{Y}_0^{(low,0),reg}$ & $\underset{(0.0070)}{12.9958}$ & $\underset{(0.0079)}{13.0059}$ & $\underset{(0.0086)}{12.9979}$ & $\underset{(0.0113)}{14.0320}$ & $\underset{(0.0099)}{14.0584}$ & $\underset{(0.0129)}{14.0638}$ \\
 \hline \vspace{-11pt} & & & & & &\\
 $\hat{Y}_0^{(low,1),reg}$ & $\underset{(0.0068)}{13.0171}$ & $\underset{(0.0075)}{13.0192}$ & $\underset{(0.0082)}{13.0118}$ & $\underset{(0.0074)}{14.0864}$ & $\underset{(0.0077)}{14.1046}$ & $\underset{(0.0093)}{14.1090}$  \\
 \hline
 \end{tabular}
 \caption{Upper and lower bounds based on the modified LSMC algorithm for different time discretizations. Standard deviations are given in brackets.}
 \label{table non-generic regression}
\end{table}
By and large, we find that the quality of the upper bound estimators $\hat{Y}_0^{(up,0),mini}$ and $\hat{Y}_0^{(up,0),reg}$, computed from the two different methods 
to obtain the coefficients for the input approximation, is almost identical. They typically vary by less than two empirical standard deviations.
The same holds true for the lower bounds $\hat{Y}_0^{(low,0),mini}$ and $\hat{Y}_0^{(low,0),reg}$. We also observe that, compared to the generic implementation,
the input lower bounds $\hat{Y}_0^{(low,0),mini}$ and $\hat{Y}_0^{(low,0),reg}$ are of the same quality as the generic lower bounds in Table \ref{table generic} 
$\hat{Y}_0^{(low,1),a^*}$ after one iterative improvement. Similarly, one improvement step of the upper bound in both non-generic cases $\hat{Y}_0^{(up,1),mini}$ and $\hat{Y}_0^{(up,1),reg}$ 
is comparable with two improvement steps in the generic setting $\hat{Y}_0^{(up,2),a^*}$. Recalling the large computational costs for the second improvement step, we observe that incorporating soft 
problem information into the function basis (here, the indicator function on the largest and second-largest asset one time step before) can significantly help to pin down the nonlinear option price 
 $Y_0$ into a rather tight confidence interval after one iteration step only (and, hence, at moderate costs). For the sake of completeness, we also report the numerical results 
after performing a second iteration step for the upper bounds in the non-generic case. While in the case of negative correlation, we obtain a further improvement and end up with 
a confidence interval of a relative width of less than 1.5 \% for $J=40$ time steps, the situation for the positive correlation case is different. Here, the theoretical improvement 
of the upper bound is offset by the additional upward bias due to the small number of inner paths. In this case, however, the relative width of the 95\% confidence interval is about 0.75\% already after one iteration step, and, thus, any further improvement seems to be unnecessary for the option pricing problem under consideration.

\end{document}